\documentclass[amssymb,12pt]{amsart}

\usepackage{latexsym}

\usepackage{hyperref}
\usepackage{amssymb,amsfonts,amsmath,amsopn,amstext,amscd,latexsym}
\usepackage{bbm}
\usepackage[all,cmtip]{xy}

\usepackage{enumerate}
\usepackage[shortlabels]{enumitem}

\newtheorem{propo}{Proposition}
\newtheorem{prop}[propo]{Proposition}

\newtheorem{lem}[propo]{Lemma}
\newtheorem{cor}[propo]{Corollary}

\newtheorem{thm}[propo]{Theorem}
\newtheorem{theorem}[propo]{Theorem}

\newtheorem{question}[propo]{Question}

\newcommand{\Aut}{{\mathrm {Aut}}}

\newcommand{\Hom}{{\mathrm {Hom}}}

\newcommand{\tr}{{\mathrm {tr}}}

\newcommand{\Spec}{{\mathrm {Spec}}}

\newcommand{\CC}{{\mathbb C}}
\newcommand{\RR}{{\mathbb R}}
\newcommand{\QQ}{{\mathbb Q}}
\newcommand{\ZZ}{{\mathbb Z}}

\newcommand{\PP}{{\mathbb P}}
\newcommand{\PGL}{{\mathrm{PGL}}}
\newcommand{\GL}{{\mathrm{GL}}}

\newcommand{\FF}{{\mathbb F}}
\renewcommand{\AA}{{\mathbb A}}

\newcommand{\gam}{\gamma}

\newcommand\SL{\mathrm{SL}}

\newcommand{\Z}{\mathbb{Z}}

\DeclareMathOperator{\PSL}{PSL}

\usepackage{todonotes}

\begin{document}

\title[Character varieties and products of trees]
{Character varieties and actions on products of trees}

\author[D. Fisher]{David Fisher}
\address{Department of Mathematics\\
    Indiana University \\
    Bloomington, IN 47405\\
    U.S.A.}
\email{fisherdm@indiana.edu}

\author[M. Larsen]{Michael Larsen}
\address{Department of Mathematics\\
    Indiana University \\
    Bloomington, IN 47405\\
    U.S.A.}
\email{mjlarsen@indiana.edu}

\author[R. Spatzier]{Ralf Spatzier}
\address{Department of Mathematics\\
    University of Michigan \\
    Ann Arbor, MI 48109\\
    U.S.A.}
\email{spatzier@umich.edu}

\author[M. Stover]{Matthew Stover}
\address{Department of Mathematics\\
    Temple University \\
    Philadelphia, PA 19122\\
    U.S.A.}
\email{mstover@temple.edu}

\thanks{DF was partially supported by NSF Grant DMS-1308291. ML was partially supported by NSF Grant DMS-1401419. MS was supported by NSF Grant DMS-1361000, and also acknowledges support from U.S. National Science Foundation grants DMS 1107452, 1107263, 1107367 "RNMS: Geometric structures And Representation varieties" (the GEAR Network). RS was partially supported by NSF Grant DMS- 1307164.  This project was begun at MSRI during the special semester on Dynamics on Moduli Spaces of Geometric Structures, the authors thank the Institute for its support and hospitality.}
\maketitle

\section{Introduction}\label{sec:Intro}

In this note we address a question that arose during the MSRI special semester on Dynamics on Moduli Spaces of Geometric Structures.

\begin{question}\label{qtn:DFRepn}
Let $\Sigma_g$ be a surface group of genus $g \ge 2$. Is there a discrete and faithful representation of $\Sigma_g$ into $\Aut(Y)$ for $Y$ a locally compact Euclidean building. Can we take $Y$ to be a finite product of bounded valence trees?
\end{question}


The question was motivated by recent results in the theory of Anosov representations \cite{KLP, GGKW} which we discuss in somewhat more detail later in this introduction.  We note that producing a free action on a product of infinite-valence trees is an easy exercise in Bass--Serre theory. Indeed, one selects the actions associated with the splittings arising from a family of essential simple closed curves that fill the surface. Moreover, it is easy to use the fact that surface groups are fully residually free to produce such an action on an infinite product of finite-valence trees.

We approach only the special case of Question \ref{qtn:DFRepn} for  products of trees and from two distinct directions.  First we attempt a construction of an action via arithmetic methods, in particular using actions on trees coming from representations into algebraic groups over fields of positive characteristic. This approach may also be interesting for building linear representations of other groups in other contexts. Second, we prove some results indicating that any positive answer to the question must be given by a highly irreducible action.  These irreducibility results obstruct some avenues for building the representations by methods of geometric topology.

Our approach to building examples leads us to:

\begin{question}\label{qtn:GlobalQtn}
Let $\Sigma_g$ be the fundamental group of a closed oriented surface of genus $g \ge 2$. Does there exist a faithful representation of $\Sigma_g$ into $\PGL_2(K)$ for some global field $K$ of characteristic $p > 0$?
\end{question}

We show that a positive answer to Question \ref{qtn:GlobalQtn} implies a positive answer to Question \ref{qtn:DFRepn}.
This is probably known, but we include a proof for the benefit of the reader.

\begin{thm}\label{thm:DiscreteFaithful}
Suppose $\Gamma$ is a finitely generated group that admits a faithful representation into $\PGL_2(K)$ for $K$ a global field of characteristic $p > 0$. Then $\Gamma$ admits a discrete and faithful representation into $\Aut(Y)$ for $Y$ a product of finite-valence trees.
\end{thm}

A result of this kind with $K$ of characteristic zero was famously exploited by Bass \cite{Bass} to show that lattices in $\SL_2(\CC)$ with a nonintegral trace admit nontrivial actions on a tree. He then used this to show that the associated hyperbolic $3$-manifold is Haken. The analogous result in characteristic zero is that a finitely generated subgroup of $\SL_2(K)$, with $K$ a number field, admits a discrete action on a finite product of hyperbolic planes, hyperbolic $3$-spaces, and trees (with at least one of these sets nonempty).

Our first approach to Question \ref{qtn:GlobalQtn} can also be thought of as a method for proving the linearity of certain finitely generated groups $\Gamma$, which we illustrate in our study of surface group representations into linear groups in characteristic $p > 0$.  This method may be of interest in other contexts, as it uses an unusual variant of ping pong for groups generated by torsion elements.  For example, we obtain:

\begin{thm}\label{thm:MainDim2}
For every $p \ge 5$ and $g \ge 2$, there exists a finite extension $K$ of $\FF_p(x, y)$ and a faithful representation from $\Sigma_g$ into $\PGL_2(K)$.
\end{thm}

Our techniques also allow us to prove the following.

\begin{thm}
\label{main}
Let $\Sigma_g$ be the fundamental group of a closed Riemann surface of genus $g$. If $p\ge 5$ and $g\ge 2$, for every field $K$ of characteristic $p$ and transcendence degree $r\ge 2$, there exists $n$ such that $\GL_n(K)$ contains
a subgroup isomorphic to $\Sigma_g$.
\end{thm}

Proving linearity results in sufficiently high transcendence degree by appealing to a universal representation was suggested to us by the construction of ``strongly dense'' free subgroups by Breuillard, Green, Guralnick, and Tao \cite{BGGT}. It is also reminiscent of the use of ``tautological representations'' by Culler and Shalen in their study of essential surfaces in knot complements \cite{CS}. The general method is as follows.

We begin by embedding $\Gamma$ into a (possibly) larger group $\Delta$, also finitely generated. Fix an adjoint simple group $G$ over $\FF_p$ and suppose that there exists an irreducible component $X$ of the character variety $\Hom(\Delta,G)/G$ for which the associated representation is generically Zariski-dense. If $K$ is the function field of $X$, then there exists a finite extension $L / K$ and a homomorphism from $\Gamma$ to $G(L)$
that represents the generic point of $X$.
Composing with the adjoint representation of $G$, we obtain a linear representation of $\Delta$, and therefore $\Gamma$, over $L$.

It remains to consider whether or not this homomorphism is faithful. To prove this, we require a subgroup $\Xi\subset G(K)$ such that for every $\delta\in \Delta\setminus \{1\}$ there exists a homomorphism $\psi\colon\Delta\to \Xi$ such that $\delta\not\in \ker \phi$. Regarding
$\psi$ as a $G$-representation of $\Delta$, it determines a point on $X$. In other words, $\Delta$ must be ``residually $\Xi$''. This implies linearity over $\FF_p(x_1,\ldots,x_r)$, where $r := \dim X$, and therefore over every characteristic $p$ field of transcendence degree $\ge r$.

For our application, $\Gamma$ will be the surface group $\Sigma_g$. Without loss of generality we can replace $\Gamma$ with a cocompact Fuchsian group $\Delta$ that contains $\Sigma_2$ (and hence $\Sigma_g$ for all $g \ge 2$). Fix $p$, and let $G=\PGL_2$. The point of introducing a group $\Delta$ which contains $\Gamma$, instead of working directly with $\Gamma$ itself, is that we can find an appropriate cocompact Fuchsian group $\Delta$ with a $\PGL_2$-character variety of dimension $2$, whereas the character variety of $\Sigma_g$ is
$6(g-1)$-dimensional. Since this dimension is the transcendence degree of the function field $K$ of our component $X$, this replacement is key to finding representations with $K$ of transcendence degree $2$ over $\FF_p$.

There do exist choices of $\Delta$ having $\PGL_2$-representation variety of dimension $1$, for example nonorientable Fuchsian groups generated by reflections in the sides of a quadrilateral. However, we could not succeed in finding a suitable $\Xi$ to implement the last step of the above method for such a $\Delta$. That is, our methods do produce surface group representations, but we cannot prove that they are faithful. As a result, we do not know if $\Sigma_g$ is \emph{ever} linear over a global field of positive characteristic.

We now state some results restricting the possible positive answers to Question \ref{qtn:DFRepn}.  We state the first only for a product of two trees, as the more general statement is somewhat technical.  Let $\Gamma$ be a torsion free hyperbolic group, let $T_1$ and $T_2$ be simplicial trees of bounded valence, and let $\rho: \Gamma \to \Aut (T_1 \times T_2)$ be a discrete faithful  representation.  Passing to a subgroup of index 2, we can assume that $\rho (\Gamma)$  preserves both trees. Let $\rho_i: \Gamma \to \Aut(T_i)$ be the induced representations.  Then we have:

\begin{theorem}\label{theorem:twotrees}
Either $\Gamma$ is free or $\rho _1$ and $\rho _2$ are faithful and indiscrete.
\end{theorem}

\noindent  One can formulate a version of this theorem for products of $n$ bounded valence simplicial trees, but the formulation is more complicated; see Theorem \ref{thm:DiscreteFaithful} below.  An easy corollary of that theorem is:

\begin{cor}\label{NoEmbed}
Suppose that $\Gamma$ is a torsion free, hyperbolic group that is not free. Then $\Gamma$ does not admit a faithful homomorphism into a finite direct product  of nonabelian free groups.
\end{cor}

\noindent This corollary is also a very special case of the main results of \cite{BHMS}, where it is shown that any $FP_{\infty}$ subgroup of $F_1 \times \cdots \times F_n$ is a finite product of free groups.  We mention the corollary here as it is an obstruction to many approaches to constructing positive answers to Question \ref{qtn:DFRepn}.  For a product of two free groups, a short argument, analyzing projections as we do,  is given in \cite[Proposition 4.1]{K}.  For contrast, those notes construct surface subgroups in right angled Artin groups that are quite closely related to a product of free groups \cite[Theorem 4.4]{K}.

We now discuss some geometric motivation behind Question \ref{qtn:DFRepn}. Representations of surface groups into real and complex Lie groups have been of significant interest since the foundational work of Fricke and Klein. However, there has also been a recent surge of activity centered around the notion of an ``Anosov representation'', which is a dynamical generalization of the discrete and faithful representations into $\PSL_2(\RR)$ associated with complete hyperbolic structures on the underlying surface. For surface groups, this was first defined by Labourie \cite{L}, and one key feature is that any Anosov representation is discrete and faithful.

Recently, \cite{KLP} (see also \cite{GGKW}) gave dynamics-free definitions that could be used to define Anosov representations into linear algebraic groups over any local field. However, one can show that the theory is empty for surface groups acting on spaces with totally disconnected Furstenberg boundary (e.g., Euclidean buildings). Elementary arguments using arithmetic groups provide a wealth of \emph{faithful} representations into algebraic groups over local fields of characteristic zero (e.g., $\PGL_2(\QQ_p)$), but these representations have large vertex stabilizers in their actions on the associated building. These constructions extend to more general locally compact groups as seen in \cite{BGSS}. However, in all contexts, \emph{discrete} and faithful representations are more difficult to construct.

We also mention that Question \ref{qtn:DFRepn} is obviously relevant to whether irreducible lattices in the isometry groups of higher rank Euclidean buildings contain surface subgroups.  This question is analogous to Gromov's famous surface subgroup conjecture for hyperbolic groups.

We now briefly outline the paper. In \S \ref{sec:Prep}, we give some general facts on character varieties. In \S \ref{sec:Sg} we concentrate on surface groups and prove Theorem \ref{main} by a direct construction. In \S \ref{sec:DiscreteFaithful} we give the proof of Theorem \ref{sec:DiscreteFaithful}.  Finally in section \ref{sec:indiscrete} we discuss Theorem \ref{theorem:twotrees}, its more general form for $n$ trees, and Corollary
\ref{NoEmbed}.

\section{Generalities on character varieties}\label{sec:Prep}

In this section, $\Delta$ will denote any finitely generated group and $G$ any adjoint simple algebraic group over a field $k$. Let $A_{\Delta,G}$ denote the coordinate ring of the affine scheme $\Hom(\Delta,G)$. The conjugation action of $G$ on $\Hom(\Delta,G)$ determines a dual action of $G$ on $A_{\Delta,G}$, and we let $B_{\Delta,G}$ denote the ring of invariants and
\[
X_{\Delta,G} := \Spec(B_{\Delta,G})
\]
be the $G$-\emph{character variety} of $\Delta$. If $\sigma$ is any representation of $G$, $m$ is an integer in $[0,\dim \sigma]$, and $\delta\in \Delta$, then the
map from $G$-representations $\rho$ of $\Delta$ to $(-1)^m$ times the $x^{\dim\sigma-m}$ coefficient of the characteristic polynomial of $\sigma(\rho(\delta))$ lies in $B_{\Delta,G}$, and we denote this function by $T(\sigma,m,\delta)$.

The natural morphism $\Hom(\Delta,G)\to X_{\Delta,G}$ is submersive \cite[Thm.~1.1]{GIT}. For the remainder of this section, fix an irreducible component $X$ of $X_{\Delta,G}$ with generic point $\eta$. This corresponds to picking a minimal prime ideal $I_0$ of $B_{\Delta, G}$ to specify the component $X$ and then taking the field of fractions $K$ of $B_{\Delta,G}/I_0$ to specify the generic point $\eta$. Then the fiber over $\eta$ is nonempty and thus has a point lying on a finite extension $L$ of the residue field $K$ of $\eta$, which determines a homomorphism
\begin{equation}\label{eq:GenericRep}
f_\eta\colon \Delta\to G(L)
\end{equation}
such that if $\bar b$ denotes the image of $b\in B_{\Delta, G}$ in $ B_{\Delta, G}/I_0\subset K\subset L$, then
for all $\delta\in \Delta$, the characteristic polynomial of $f_\eta(\delta)$ is
$$\sum_{m=0}^{\dim \sigma} (-1)^m\overline{T(\sigma,m,\delta)} x^{\dim\sigma-m}.$$
The following gives a sufficient condition for $f_\eta$ to be faithful when $k$ has positive characteristic.

\begin{prop}\label{prop:GLinear}
Suppose that $\Delta$ has no $p$-torsion. For every nontrivial $\delta \in \Delta$, suppose that there exists a field extension $M/\FF_p$ and a homomorphism $g\colon \Delta\to G(M)$ associated with a point in $X(M)$ such that $g(\delta)$ does not have order $p^t$ for any integer $t \ge 0$. Then the homomorphism $f_\eta$ defined in \eqref{eq:GenericRep} is injective.
\end{prop}

\begin{proof}
Suppose that $\delta\in\ker f_{\eta}$ and let $\sigma$ denote the adjoint representation of $G$.
Since the characteristic polynomial of the trivial matrix is $(x-1)^{\dim G}$, for all integers $0\le m\le \dim G$, the function
\[
T(\sigma,m,\delta) - \binom{\dim G}{m}
\]
vanishes at $\eta$, and hence vanishes on all of $X$ (i.e., maps to zero under $b\mapsto \bar b$ and therefore under
$B_{\Delta,G}\to B_{\Delta,G}/\mathfrak{p}$ for any prime ideal $\mathfrak{p}$ containing $I_0$). Since $g$ belongs to $X(M)$ and $\delta$ is in the kernel of $f_{\eta}$,
the characteristic polynomial of $\sigma(g(\delta))$ is $(x-1)^{\dim G}$, and therefore $\sigma(g(\delta))$ is unipotent.
If $p^t > \dim G$, then $\sigma(g(\delta^{p^t})) = 1$, which implies $g(\delta^{p^t})=1$, and this is a contradiction if $\delta$ is nontrivial.
\end{proof}

\medskip
\noindent
\textbf{Remark.}\ A similar statement holds in characteristic zero with the same proof, where we instead assume that $X(M)$ contains a point associated with a representation $g$ for which $g(\delta)$ is not unipotent. This type of argument appears in Culler and Shalen's celebrated paper on character varieties of fundamental groups of $3$-manifolds \cite{CS}. There $\Delta$ is the fundamental group of a one-cusped hyperbolic $3$-manifold $M$ of finite volume, $X$ is the component of the $\SL_2(\CC)$ character variety containing a representative for the discrete and faithful representation of $\Delta$ associated with the complete structure on $M$, and $g$ is an element of $\Delta$ representing a nontrivial loop on the boundary torus.
\medskip




At the expense of increasing the dimension of the representation, we obtain that $\Delta$ is linear over a  purely transcendental field of degree $r = \dim(X)$.

\begin{cor}
\label{LinearB}
If $p$ is a prime, $G$ is adjoint simple algebraic group over $\FF_p$, $\Delta$ is a finitely generated group without $p$-torsion, and $X_{\Delta,G}$ has an $r$-dimensional closed subscheme satisfying the assumptions of Proposition \ref{prop:GLinear}, then $\Delta$ is linear over $\FF_p(x_1,\ldots,x_r)$.
\end{cor}

\begin{proof}
We showed that there is an injective homomorphism to $G(L)$, with $L$ a finite extension of the function field $K$ of $X$. Note that $K$ is a finitely generated extension of $\FF_p$ of transcendence degree $r$. Under the adjoint representation of $G$, we then obtain an injective homomorphism from $\Delta$ to $\GL_n(L)$, where $n := \dim G$. Choosing $r$ algebraically independent elements $x_1,\ldots,x_r\in L$, we obtain a purely transcendental subfield $k := \FF_p(x_1,\ldots,x_r)$ of which $L$ is a finite extension, say of degree $l$. Then $\Delta$ can be realized as a subgroup of $\GL_{nl}(k)$.
\end{proof}

\medskip
\noindent
\textbf{Remark.} In the above, we can replace $X$ with any irreducible closed subscheme of $X_{\Delta, G}$. In fact, embedding a group $\Gamma$ in a larger group $\Delta$ does precisely this by considering representations of $\Gamma$ that extend up to $\Delta$.
\medskip

\section{Surface groups}\label{sec:Sg}

\newcommand{\at}{A}
\newcommand{\bt}{B}
\newcommand{\ct}{C}
\newcommand{\dt}{D}
\newcommand{\z}{Z}

For this section we fix
\[
\Delta:=\langle a,b,c,d\,|\,a^3,\,b^2,\,c^2,\,d^3,\,abcd\rangle.
\]
Then $\Delta$ is a compact oriented Fuchsian group in which every element of finite order is conjugate to $a^{\pm1}$, $b$, $c$, $d^{\pm1}$, or $1$. There is a homomorphism from $\Delta$ onto $S_3$ sending $a$ and $d$ to $(123)$, $b$ to $(12)$, and $c$ to $(23)$. The kernel $\Gamma$ is of index $6$ and contains no nontrivial torsion element. Its Euler characteristic is
\[
\chi(\Gamma)=6\chi(\Delta) = 6(2-4+1/3+1/2+1/2+1/3) = -2.
\]
and it follows that it is isomorphic to $\Sigma_2$. Consequently, linear representations of $\Delta$ determine representations of $\Sigma_g$ for all $g \ge 2$.

Let
\[
\tilde \Delta = \langle \at,\bt,\ct,\dt,\z\,|\,\at^3 = \bt^2 =\ct^2 =\dt^3 = \z,\,\z^2 = \at\bt\ct\dt = 1\rangle.
\]
The map $(\at,\bt,\ct,\dt,\z)\mapsto(a,b,c,d,1)$ gives a surjective homomorphism $\tilde \Delta\to \Delta$ with central kernel $\langle \z\rangle \cong \ZZ/2\ZZ$.
Every element in $\tilde \Delta$ can be written as $\z^i$ times a word in the symbols $\at^{\pm1}$, $\bt$, $\ct$, $i\in\{0,1\}$.  Replacing the left hand side by the right in
\begin{equation}
\label{one}
\ct \at \bt \ct=  \bt \at^{-1} \ct \bt \at^{-1}
\end{equation}
and
\begin{equation}
\label{two}
\ct \bt \at^{-1} \ct= \at \bt \ct \at \bt
\end{equation}
we can reduce the number of symbols of type $\ct$.  We say a word is \emph{reduced} if no two consecutive symbols are equal or inverse to one another
and all  replacements of type \eqref{one} and \eqref{two} have been made.  Up to powers of $Z$,
the inverse of a word in $\at$, $\at^{-1}$, $\bt$, $\ct$ is obtained by reversing the order of symbols and interchanging $\at$ and $\at^{-1}$; in particular,
this process sends reduced words to reduced words.  A word is \emph{cyclically reduced} if every word obtained from it by cyclic permutation of symbols in
reduced.  Note that such a permutation does not affect the conjugacy class and therefore does not affect character values.

\begin{prop}
Let $K$ be an algebraically closed field of characteristic $p\ge 5$. For all $\delta\in\tilde\Delta$, there exists
$T_\delta\in \ZZ[u,v,w]$ such that if
$\rho\colon \tilde \Delta\to \SL_2(K)$ is any representation such that $\rho(\at)$, $\rho(\bt)$, $\rho(\ct)$, $\rho(\dt)$ have
orders $6$, $4$, $4$, and $6$ respectively, then
\[
\tr(\rho(\delta)) =  T_\delta(\tr(\rho(\bt\ct)), \tr(\rho(\ct \at)), \tr(\rho(\at \bt))).
\]
\end{prop}

\begin{proof}
Let $\chi$ denote the character $\tr\circ \rho$.
It suffices to prove that
\[
\chi(w)=  T_\delta(\chi(\bt\ct), \chi(\ct \at), \chi(\at \bt))
\]
for all reduced words $w$.
We prove this by induction on the length of $w$.  As $\chi(z\delta) = -\chi(\delta)$, it suffices to consider words without occurrence of $Z$,
and moreover, we may assume $w$ is cyclically reduced.
For length $0$, $\chi(w) = 2$.  For length $1$, $\chi(\at) = 1$, $\chi(\bt) = \chi(\ct)=0$.
For length $2$ the character value does not depend on the order of the symbols.  Unless $\at^{-1}$ appears, there is nothing to check.
Moreover,
\begin{align*}
\chi(\at^{-1}\bt) &= \chi((\at^{-1}\bt)^{-1}) = -\chi(\bt \at),\\
\chi(\at^{-1}\ct) &= \chi((\at^{-1}\ct)^{-1}) = -\chi(\ct \at).
\end{align*}
The cyclically reduced words of length 3, up to inversion, cyclic permutation, and multiplication by $\z$,
are $\at\bt\ct$ and $\at\ct\bt$.  We have
\[
\chi(\at\bt\ct) = \chi(\dt^{-1}) = \chi(\dt) = 1.
\]
From the $\SL_2$ identity
\begin{equation}
\label{identity}
\tr(x)\tr(y) = \tr(xy) + \tr(x^{-1}y),
\end{equation}
we deduce
\begin{align*}\chi(\at \ct\bt) &= \chi(\at)\chi(\ct\bt) - \chi(\at^{-1}\ct\bt) = \chi(\ct\bt) - \chi(\at^{-1}\ct\bt) \\
&= \chi(\bt\ct) - \chi(\bt\ct\at)= \chi(\bt\ct) - 1.
\end{align*}

For words cyclically reduced words $w$ of length $l \ge 4$, we may assume, after cyclic permutation of the symbols of $w$ if necessary,
that the first symbol of $w$ either  coincides with or is inverse to the $k$th symbol for some $k  \in [2,l-1]$.
Defining $x$ and $y$ to be the product of the first $k-1$ symbols of $w$ and the
remaining $l-k+1$ symbols respectively, we can then use \eqref{identity} to express $\chi(w)$ as a polynomial with $\ZZ$-coefficients in terms
of $\chi$-values of words strictly shorter than $w$.
\end{proof}

\begin{thm}
\label{tDelta}
The component of the character variety $\Hom(\tilde\Delta,\SL_2)/\SL_2$ containing the representations that are faithful when restricted to
$\langle \at\rangle$, $\langle \bt\rangle$, $\langle \ct\rangle$, and $\langle \dt\rangle$ is birationally equivalent to the affine cubic surface
\[
S := \Spec\big(K[u,v,w]/(uvw+u^2+v^2+w^2-u-2)\big).
\]
\end{thm}

\begin{proof}
We consider the morphism $\phi\colon \Hom(\tilde\Delta,\SL_2)/\SL_2\to \AA^3$ mapping
\[
\chi\mapsto (\chi(\bt\ct),\chi(\ct\at),\chi(\at\bt)).
\]
For $K$ not of characteristic $2$ or $3$, the conditions $\chi(\at)=\chi(\at\bt\ct)=1$ and $\chi(\bt)=\chi(\ct)=0$ cut out the open and closed subscheme $X$
of $\Hom(\tilde\Delta,\SL_2)/\SL_2$ consisting of representations $\rho$ faithful on the finite cyclic groups generated by
$\at$, $\bt$, $\ct$, $\dt$.
The following identity for triples $(X,Y,Z)$ of $2\times 2$ matrices can  be verified easily:
\begin{align*}\tr(&X)\tr(Y)\tr(Z)\tr(XYZ) + \tr(YZ)\tr(ZX)\tr(XY) + \tr(XYZ)^2 \\
&+  \sum \det X \tr(YZ)^2 + \sum \tr(X)^2\det YZ - 4\det XYZ\\
&-\sum \det X\tr(Y)\tr(Z)\tr(YZ) - \sum \tr(X)\tr(YZ)\tr(XYZ)=0,\\
\end{align*}
where each indicated sum is a symmetric sum of $3$ terms (including the specified one) given by cycling permuting $X, Y$ and $Z$.
This implies the following identity for triples $(x,y,z)$ in $\SL_2$:
\begin{align*}\tr(x)\tr(y)&\tr(z)\tr(xyz) + \tr(yz)\tr(zx)\tr(xy) + \tr(xyz)^2 \\
&\qquad\qquad\qquad\qquad\qquad+  \sum \tr(yz)^2 + \sum \tr(x)^2 \\
&=\sum \tr(y)\tr(z)\tr(yz) + \sum \tr(x)\tr(yz)\tr(xyz)+ 4.\\
\end{align*}

It follows that the restriction of $\phi$ to $X$ maps to $S$.  To see this recall that we map $A \rightarrow x$, $B \rightarrow y$ and $C\rightarrow z$ and that $\chi(\at) = 1$, $\chi(\bt) = \chi(\ct)=0$ and $\chi(\at\bt\ct)=1$ and let the variable $u= \chi(BC)$, $v = \chi(CA)$ and $w=\chi(AB)$.

Since $S$ is irreducible, to show that $\phi|_X$ is dominant
it suffices to prove the image is $2$-dimensional.  We can do this by projecting the image onto the $u$-$w$ plane.  Setting
\[
\at = \begin{pmatrix} q&r \\ \frac{-1+q-q^2}r&1-q\end{pmatrix},\, \bt = \begin{pmatrix} i&0\\0&-i\end{pmatrix},\, \ct = \begin{pmatrix}s&1\\-1-s^2&-s \end{pmatrix},
\]
we obtain a triple in $X$ which maps to $((2q-1)i,2si)$ as long as there exists a solution $r$ to the equation
\[
i(s^2+1)r - \frac{s(q^2-q+1)}r + R(s,q) = 0
\]
for an explicit polynomial $R$.  As long as $s$ is not a $4$th root of unity or $q$ is not a $6$th root of unity, this has at least one solution $r$.

Finally, $\phi|_X$ is an isomorphism at every point in $X$ corresponding to an absolutely irreducible representation \cite{Nakamoto}.
As long as $\chi(\bt\ct) \not\in \{\pm 2\}$, $\rho(\bt)$ and $\rho(\ct)$ cannot lie in a common Borel subgroup of $\SL_2$, and it follows that
$\chi$ is absolutely irreducible.
\end{proof}

Our goal now is to show that the representation of $\tilde\Delta$ on the function field of the cubic surface $S$ from Theorem \ref{tDelta} is faithful. By Proposition \ref{prop:GLinear}, this provides a faithful representation of $\Delta$ into $\PGL_2$ over a field of transcendence degree $2$ and characteristic $p$, which completes the proof of Theorem \ref{thm:MainDim2}. To achieve this, we use Proposition \ref{prop:GLinear} and the following result, which shows that $\Delta$ is residually $\Xi$ for $\Xi$ a virtually free group (in fact, $\Xi$ is isomorphic to the modular group $\PSL_2(\ZZ)$).

\begin{prop}
\label{Residual}
For every nontrivial element $\delta\in \Delta$ there exists $N$ such that for every $n>N$, $\delta$ is not in the kernel of the homomorphism
\[
\psi_n\colon \Delta\to \Xi := \langle x,y \mid x^3 , y^2 \rangle
\]
defined by
\[
a \mapsto x,\,b \mapsto y,\, c \mapsto (xy)^n y^{-1} (xy)^{-n},\, d \mapsto (xy)^n x^{-1} (xy)^{-n}.
\]
\end{prop}

\begin{proof}
Let $S(a,b)$ denote the set of finite strings $\alpha$ in the symbols $a$, $a^{-1}$, $b$ such that an $a$ or $a^{-1}$ can only be followed by a $b$ and vice versa.  Let $S(c,d)$ denote the set of finite strings $\gamma$ in  $c$, $d$, and $d^{-1}$
such that $d$ or $d^{-1}$ can only be followed by $c$, and vice versa, and moreover, every $\gamma$ is nonempty, and $\gamma$ cannot begin or end with the substring $cd$ or $d^{-1} c$.
For any element $\alpha\in S(a,b)$ (resp. $\gamma\in S(c,d)$), we denote by $\bar \alpha$ (resp.\ $\bar\gamma$) the element of $\Xi$ obtained by substituting $x$ for $a$ and $y$ for $b$ (resp.\ $y^{-1}$ for $c$ and $x^{-1}$ for $d$).
Every element $\delta \in \Delta$ can be written uniquely as a word
\[
\alpha_0 \gamma_1\alpha_1\cdots\gamma_k \alpha_k,
\]
for some $k\ge 0$, some $\alpha_i\in S(a,b)$ and some $\gamma_j\in S(c,d)$.
Thus
\[
\psi_n(\delta) = \bar\alpha_0 (xy)^n \bar \gamma_1 (yx^{-1})^n\bar\alpha_1 \cdots (xy)^n \bar \gamma_k (yx^{-1})^n\bar\alpha_k.
\]
As $\gamma_i$ cannot begin with $d^{-1} c$ or end with $cd$, at most one symbol pair cancels in $(xy)^n \bar \gamma_i$, and at most one pair cancels in $\bar\gamma_i (yx^{-1})^n$.  Therefore, if $n$ is larger than the maximum length of the $\alpha_i$, the expression as a whole is nontrivial in $\Xi$ as claimed.
\end{proof}

\medskip
\noindent
\textbf{Remark.} The group $\Delta$ is the free product of $\Xi$ with itself amalgamated over the subgroup generated by $a b = (c d)^{-1}$. Geometrically this splits the sphere with four cone points of orders $3,2,2,3$ along the curve dividing the cone points into two sets of $2,3$. The map $\psi_0$ is the natural projection onto $\Xi$ from this free product with amalgamation, and $\psi_n$ for $n \ge 1$ is the composition of a Dehn twist of order $n$ around $a b$ with $\psi_0$. The above algebraic calculation means that the curve obtained from $\delta$ after a sufficiently high power of the Dehn twist must map to a nontrivial reduced word in the group $\Xi \cong \ZZ/2 * \ZZ/3$ under $\psi_0$. This strategy was first used by Baumslag to show that surface groups are residually free \cite{Ba}.

\begin{prop}
\label{Hausdorff}
For all $p\ge 3$, there exists an injective homomorphism  $\Xi\to \PGL_2(\FF_p(t))$.
\end{prop}

\begin{proof}
We use the construction of Hausdorff \cite[Anh.~pp.~469--472]{H}, which we recast in a somewhat more conceptual form.
Consider the homomorphism $\iota$ defined by
\[
x\mapsto \left[\begin{pmatrix}0&1\\ -1&-1\end{pmatrix}\right],\ y\mapsto\left[\begin{pmatrix}0&1\\ t&0\end{pmatrix}\right],
\]
where $[M]$ denotes the image of $M\in \GL_2(\FF_p(t))$ in $\PGL_2(\FF_p(t))$.  To prove $\iota$ is injective, we consider the action
of $\Xi$ on $\PP^1(\FF_p(t))$ which it determines.  There is a natural map $\PP^1(\FF_p(t))\to \PP^1(\FF_p)$ characterized by
\[
(P(t):Q(t))\mapsto (P(0):  Q(0)),
\]
where $P(t)$ and $Q(t)$ are relatively prime polynomials.  Whereas $\iota(x)$ induces a well-defined operation on all of $\PP^1(\FF_p)$,
$\iota(y)$ maps $\PP^1(\FF_p)\setminus \{\infty\}$ to $\infty$ but is not well-defined at $\infty$. Since the $\iota(x)$-orbit of $\infty$ is $\{0,-1,\infty\}$, it follows that for every nontrivial $g$, $\iota(g)$ maps $1$ to $0$, $-1$, or $\infty$. Thus $\iota$ is injective.
\end{proof}

We conclude by proving Theorem~\ref{main}. There is a natural morphism
$\Hom(\tilde \Delta,\SL_2)\to \Hom(\Delta,\PGL_2)$ which is surjective with finite fibers
By Theorem~\ref{tDelta}, the generic point $\eta$ of $S$ maps to the generic point of a $2$-dimensional component of
$X_{\Delta,\PGL_2}$. This component contains points representing all homomorphisms $\Delta\to \PGL_2$
which lift to homomorphisms $\tilde \Delta\to \SL_2$ where $A$, $B$, $C$, and $D$ map to points of order
$6$, $4$, $4$, and $6$ respectively, i.e., all points representing homomorphisms where $a$ and $b$
maps to points of order $3$ and $2$, respectively.  By Corollary~\ref{LinearB}, Proposition~\ref{Residual}, and Proposition~\ref{Hausdorff}, this implies the theorem.

\section{The proof of Theorem \ref{thm:DiscreteFaithful}}\label{sec:DiscreteFaithful}

Let $K$ be a global field of characteristic $p > 0$. Then $K$ is the function field of a unique smooth projective curve $C$ over $\FF_p$. Each place $\nu$ of $K$ determines a nonarchimedean local field $K_\nu$, the completion of $K$ with respect to the topology induced by the equivalence class of valuations determined by $\nu$. Let $\mathcal{O}_\nu$ be the valuation ring of $K_\nu$ and
\[
\AA \subset \prod_\nu K_\nu
\]
be the adele ring of $K$, i.e., the restricted direct product of the $K_\nu$ with respect to the $\mathcal{O}_\nu$. The diagonal embedding of $K$ into $\AA$ is well-known to be discrete \cite[\S II.14]{CF}, and this extends to a discrete embedding of $\PGL_2(K)$ into $\PGL_2(\AA)$.

Therefore, if $\rho : \Gamma \to \PGL_2(K)$ is a faithful representation, then $\rho$ extends to a discrete representation of $\rho_\AA\colon \Gamma \to \PGL_2(\AA)$. Since $\Gamma$ is finitely generated, this in fact determines a discrete representation of $\Gamma$ into
\[
\rho_S\colon \Gamma \to G = \prod_{\nu \in S} \PGL_2(K_\nu)
\]
for some finite nonempty collection $S$ of places of $K$. Indeed, each generator $\gam$ of $\Gamma$ in a fixed generating set lies in $\PGL_2(\mathcal{O}_\nu)$ for all but a finite set $S_\gam$ of places. Since the image of $\Gamma$ under $\rho_\AA$ is discrete, each $S_\gam$ must be nonempty, and we take $S$ to be the union of the $S_\gam$ over our finite generating set.

We refer the reader to \cite{Serre} for the definition of the tree $T_\nu$ associated with the group $\PGL_2(K_\nu)$. Recall that $T_\nu$ is a $(q + 1)$-regular tree, where $q$ is the cardinality of the residue field of $K$. That the image of $\hat{\rho}$ is discrete means that $\rho_S(\Gamma)$ acts discretely on
\[
X = \prod_{\nu \in S} T_\nu.
\]
In particular, we have shown that $\Gamma$ acts discretely and faithfully on a finite product of finite-valence trees, which proves Theorem \ref{thm:DiscreteFaithful}.

\medskip
\noindent
\textbf{Remark.} In general, one has little control over the number of trees. However, when the representation is constructed from a curve on the $\SL_2(\FF_p)$-character variety of $\Gamma$, there is a bound. Let $X$ be an affine curve of characters of $\Gamma$, and assume that the generic point of $X$ determines a representation of $\Gamma$ into $\SL_2(K)$, where $K$ is the function field of $X$. Let $C$ be the smooth complete curve associated with $X$ and $f\colon C \to X$ be the normalization, which is a birational morphism. The points
\[
C \smallsetminus f^{-1}(X)
\]
are sometimes called the \emph{ideal points} of $X$. The image of $\Gamma$ in $\SL_2(K_\nu)$ lies in $\SL_2(\mathcal{O}_\nu)$ for every $\nu$ associated with a point in $f^{-1}(X)$. Indeed, the translation length on $T_\nu$ is associated with the valuation of the trace, and the character must have a pole at the point on $C$ associated with $\nu$ in order to have nontrivial translation length. However, characters are finite-valued on the affine part of $X$, so the set $S$ described above consists only of ideal points.

\section{Actions on products of trees}
\label{sec:indiscrete}

Our goal is to prove the following generalization of Theorem \ref{theorem:twotrees}:

\begin{thm}\label{NoAction}
Suppose that $\Gamma$ is a torsion free hyperbolic group that is not free. Let
\[
X = T_1 \times \cdots \times T_n
\]
be a product of finite-valence trees, $G_i = \mathrm{Aut}(T_i)$, and $G = \prod G_i$. Let $\sigma_i$ denote the projection of $G$ onto $G_i$. If $\rho : \Gamma \to G$ is a discrete and faithful representation, then there are at least two $i$ such that $\rho_i = \sigma_i \circ \rho$ is faithful and has indiscrete image. Moreover, suppose $\rho_1, \dots, \rho_r$ are faithful representations and the other $\rho_i$ are not. Then the representation
\[
\rho_1 \times \cdots \times \rho_r : \Gamma \to G_1 \times \cdots \times G_r
\]
is discrete and faithful.
\end{thm}


Before embarking upon the proof, we prove Corollary \ref{NoEmbed} from the introduction:

\begin{proof}
Consider a discrete and faithful action of $F_i$ on a tree $T_i$. Then the composition
\[
\Gamma \to \prod F_i \to \prod \mathrm{Aut}(T_i)
\]
is discrete and faithful, which contradicts Theorem \ref{NoAction}.
\end{proof}

We will repeatedly use the following theorem of Serre:

\begin{thm}
A torsion free, discrete group of isometries of a (single) tree is free.
\end{thm}

\noindent In the proof of the theorem, we will also use the following lemma.

\begin{lem}\label{NIHyperbolic}
Let $\Gamma$ be a torsion-free hyperbolic group. Then any two nontrival normal subgroups of $\Gamma$ have nontrivial intersection.
\end{lem}

\begin{proof}
Let $K$ and $L$ be normal subgroups of $\Gamma$. Pick $x \in K$ and $y \in L$ nontrivial elements, which we can assume to be distinct, else we would be done. In fact, we can assume that $x^n \neq y^m$ for any $n,m \in \Z$. Then it is a consequence of ping-pong in hyperbolic groups that there are positive $n,m$ such that $\langle x^n, y^m \rangle$ is a free group of rank two. However, we also have
\[
1 \neq [x^n, y^m] \in K \cap L,
\]
so the intersection of $K$ and $L$ is nontrivial.
\end{proof}

We now prove Theorem \ref{NoAction}.

\begin{proof}[Proof of Theorem \ref{NoAction}]
Suppose that $\rho_1, \dots, \rho_r$ are faithful, $r \ge 0$, and that $\rho_{r + 1}, \dots, \rho_n$ are not. We also assume that $r < n$, else we would be done. Let $K_i$ be the kernel of $\rho_i$. Lemma \ref{NIHyperbolic} implies that
\[
K = \bigcap_{i = r+1}^n K_i
\]
is a nontrivial normal subgroup of $\Gamma$. In particular, we see that $r \ge 1$, since otherwise $\rho$ would have a nontrivial kernel.

Perhaps after replacing $\Gamma$ with a subgroup of finite index, we can assume that there is a vertex $p_i \in T_i$ with nontrivial stabilizer $\Delta_i \in \Gamma$ for $1 \le i \le r$. Indeed, $\Gamma$ is not free, so the image of the faithful representation $\rho_i$ is indiscrete. It follows that $\Gamma$ has a fixed point on $T_i$, which we can assume to be a vertex after passage to a finite index subgroup.

Consider $x \in K$ and $y \in \Delta_i$. Our goal is to find an element in $K \cap \Delta_i$, so we can assume $x$ and $y$ do not lie in a cyclic subgroup.  As in the lemma above, using ping-pong and passing to powers of $x$ and $y$, we may assume they generate a nonabelian free group.   Then
\[
[x, y^k] = x (y^k x^{-1} y^{-k}) \in K.
\]
for every $k \in \Z$. We claim that there is some $k \ge 1$ such that we also have $[x, y^k] \in \Delta_i$, so then $[x, y^k] \in K \cap \Delta_i$. To prove this, consider the action of $\Gamma$ on $T_i$ induced by $\rho_i$, where we have
\[
(y^k x^{-1} y^{-k}) \cdot p_i = (y^k x^{-1}) \cdot p_i.
\]
Since $y^k$ fixes $p_i$, the points $\{(y^k x^{-1}) \cdot p_i\}$ are all vertices of $T_i$ at distance $d_{T_i}(p_i, x \cdot p_i)$ from $p_i$. There are only finitely many such points, so there must be $k_1, k_2 \in \Z$ with $k_1 \neq k_2$ such that
\[
(y^{k_1} x^{-1}) \cdot p_i = (y^{k_2} x^{-1}) \cdot p_i,
\]
which implies that
\[
(y^{k_1 - k_2} x^{-1}) \cdot p_i = x^{-1} \cdot p_i,
\]
and hence that
\[
x y^k x^{-1} \cdot p_i = p_i
\]
with $k = k_1 - k_2$, and it follows that $[x, y^k] \in \Delta_i$, as claimed.  As we assumed that $x$ and $y$ generate a free subgroup, we also know $[x,y^k] \neq 1$.


Now suppose that $r = 1$. Then there is exactly one $i$ such that $\rho_i$ is faithful and we only have $\Delta_i$ for $i = 1$. Then the element $[x, y^k]$ constructed above (or $x^n$ in the degenerate case where $x$ and $y$ have a common power) has a global fixed point
\[
p = (p_1, \dots, p_n) \in X,
\]
where $p_i$ is any point on $T_i$ for $2 \le i \le n$. This proves that $r \ge 2$, i.e., that at least two $\rho_i$ are faithful, and hence indiscrete.

Now consider the representation
\[
\rho_1 \times \cdots \rho_r : \Gamma \to G_1 \times \cdots \times G_r.
\]
It is evidently faithful, since each $\rho_i$ is faithful by assumption. We must prove that the image is discrete. If not, then we could find $y \in \Gamma$ with a fixed point
\[
p = (p_1, \dots, p_r) \in T_1 \times \cdots \times T_r.
\]
Let $\Delta$ be the stabilizer of $p$ in $\Gamma$. Then the above argument applied to $\Delta$ in place of $\Delta_i$ produces a nontrivial element $\gamma \in \Delta \cap K$. However, $\gamma$ then fixes the point
\[
(p_1, \dots, p_r, p_{r+1}, \dots, p_n) \in T_1 \times \cdots \times T_n,
\]
where $p_i \in T_i$ is arbitrary for $i \ge r + 1$. This violates discreteness of $\rho$, which shows that $\rho_1 \times \cdots \times \rho_r$ must be discrete. This completes the proof of the theorem.
\end{proof}



\medskip
\noindent
\textbf{Remark.}
These arguments apply to a much wider class of groups. For example, Theorem \ref{NoAction} holds for Gromov hyperbolic groups with torsion but with trivial center. In nonuniform lattices in rank one Lie groups, one finds an appropriate loxodromic element $x \in K$ (e.g., using the fact that the limit set of $K$ is the limit set of $\Gamma$), and uses the fact that a loxodromic element is centralized only by the maximal cyclic subgroup of $\Gamma$ containing it, to prove the theorem with a variant of the above argument. Instead of showing that $\langle x, y \rangle$ is free of rank one or two, one must argue with ping-pong that it is either cyclic or that one can pass to powers such that $\langle x^n, y^m \rangle$ is free of rank two, and be careful of elliptic and parabolic elements.  For higher rank lattices, there is an entire literature that shows much stronger and more general results than Theorem \ref{NoAction}.


\end{document}